\documentclass{amsart}

\newtheorem{thm}{Theorem}[section]
\newtheorem{cor}[thm]{Corollary}
\newtheorem{lem}[thm]{Lemma}

\theoremstyle{definition}

\newtheorem{rem}[thm]{Remark}
\newtheorem{exa}[thm]{Example}

\numberwithin{equation}{section}

\numberwithin{equation}{section}

\usepackage{pinlabel}  

\usepackage{tikz}
\usepackage{tikz-cd}
\usepackage{amsmath}
\usepackage{amsfonts}
\usepackage{amssymb}
\usepackage{mathrsfs}
\usepackage{mathtools}
\usepackage{color}

\usepackage{comment}
\usetikzlibrary{automata,topaths}

\newcommand*{\longhookrightarrow}{\ensuremath{\lhook\joinrel\relbar\joinrel\rightarrow}}

\def\gbb{{\mathbb{G}}}

\def\qbb{{\mathbb{Q}}}

\def\zbb{{\mathbb{Z}}}
\def\acal{{\mathcal{A}}}
\def\bcal{{\mathcal{B}}}
\def\ccal{{\mathcal{C}}}

\def\ncal{{\mathcal{N}}}

\def\qcal{{\mathcal{Q}}}

\def\wcal{{\mathcal{W}}}

\def\im{{\mathrm{Im}}}
\def\ker{{\mathrm{Ker}}}

\numberwithin{equation}{section}
\newcommand{\rmsph}{{\mathrm{SPh}}}
\newcommand{\rmph}{{\mathrm{Ph}}}

\newcommand{\rmitem}[1]{\item[{\rm{(#1)}}] }

\begin{document}

\title{Phantom maps and fibrations}


\author{Hiroshi Kihara}
\address{Center for Mathematical Sciences, University of Aizu, 
	Tsuruga, Ikki-machi, Aizu-Wakamatsu City, Fukushima, 965-8580, Japan}
\email{(kihara@u-aizu.ac.jp)}




\date{}

\dedicatory{}



		\maketitle
		
		\begin{abstract}
			Given pointed $CW$-complexes $X$ and $Y$, $\rmph(X, Y)$ denotes the set of homotopy classes of phantom maps from $X$ to $Y$ and $\rmsph(X, Y)$ denotes the subset of $\rmph(X, Y)$ consisting of homotopy classes of special phantom maps. In a preceding paper, we gave a sufficient condition such that $\rmph(X, Y)$ and $\rmsph(X, Y)$ have natural group structures and established a formula for calculating the groups $\rmph(X, Y)$ and $\rmsph(X, Y)$ in many cases where the groups $[X,\Omega \widehat{Y}]$ are nontrivial. In this paper, we establish a dual version of the formula, in which the target is the total space of a fibration, to calculate the groups $\rmph(X, Y)$ and $\rmsph(X, Y)$ for pairs $(X,Y)$ to which the formula or existing methods do not apply. In particular, we calculate the groups $\rmph(X,Y)$ and $\rmsph(X,Y)$ for pairs $(X,Y)$ such that $X$ is the classifying space $BG$ of a compact Lie group $G$ and $Y$ is a highly connected cover $Y' \langle n \rangle$ of a nilpotent finite complex $Y'$ or the quotient $\gbb / H$ of $\gbb = U, O$ by a compact Lie group $H$.
		\end{abstract}
		\section{Introduction} \label{section1}
		Given two pointed $CW$-complexes $X$ and $Y$, a map $f: X \longrightarrow Y$ is called a {\em{phantom map}} if for any finite complex $K$ and any map $h: K \longrightarrow X$, the composite $fh$ is null homotopic.
		The concept of a phantom map, which is one of the most important concepts in homotopy theory, is essential to understanding maps with infinite dimensional sources (\cite{McGibbon95,Roitberg94}).
		
		Let $\rmph(X, Y)$ denote the subset of $[X, Y]$ consisting of homotopy classes of phantom maps, and let $\rmsph(X, Y)$ denote the subset of $\rmph(X,Y)$ consisting of homotopy classes of {\em{special phantom maps}}, defined by the exact sequence of pointed sets
		\[0 \longrightarrow \rmsph(X,Y) \longrightarrow \rmph(X,Y) \overset{e_{Y\sharp}}{\longrightarrow} \rmph(X,\check{Y}), \]
		where $e_Y: Y \longrightarrow \check{Y} = {\prod}_{p}\, Y_{(p)}$ is a natural map called the local expansion (cf. \cite[p.~150]{Roitberg94}). The target $Y$ is usually assumed to be nilpotent of finite type.
		
		Previous calculations of $\rmph(X, Y)$ had generally assumed that $[X, \Omega\hat{Y}]$ is trivial, in which case generalizations of Miller's theorem are directly applicable, and calculations of $\rmsph(X, Y)$ had rarely been reported (see \cite[Section 1]{phantom}). In \cite{phantom}, we gave a sufficient condition such that $\rmph(X, Y)$ and $\rmsph(X, Y)$ have natural group structures, which is much weaker than the conditions obtained by Meier and McGibbon (\cite{Meier75}, \cite[Theorem 4]{McGibbon93}), and established a formula which enables us to calculate not only $\rmph(X,Y)$ but also $\rmsph(X,Y)$ in many cases where the groups $[X,\Omega \widehat{Y}]$ are nontrivial (see Section 2.1 for these results, which are recorded as Theorems \ref{fundamental} and \ref{calculational.base}). 
		\par\indent
		In this paper, we establish a dual version of the formula and apply it to calculate the groups $\rmph(X, Y)$ and $\rmsph(X, Y)$ for pairs $(X,Y)$ with $[X,\Omega \widehat{Y}]\neq 0$ to which the formula or existing methods do not apply. 
		\par\indent
		We state the main results of this paper more precisely.
		
		Let $\ccal\wcal$ denote the category of pointed connected $CW$-complexes and homotopy classes of maps and let $\mathcal N$ denote the full subcategory of $\mathcal{CW}$ consisting of nilpotent $CW$-complexes of finite type. Let $\qcal$ be the full subcategory of $\ccal\wcal^{\mathrm{op}} \times \ncal$ consisting of $(X, Y)$ such that for each pair $i,j > 0$, the rational cup product on $H^{i}(X; \qbb) \otimes H^{j}(X; \qbb)$ or the rational Whitehead product on $(\pi_{i+1} (Y) \otimes \qbb) \otimes (\pi_{j+1}(Y) \otimes \qbb)$ is trivial. Then, $\rmph(X, Y)$ and $\rmsph(X, Y)$ have natural divisible abelian group structure for $(X, Y) \in \qcal$ (see Theorem \ref{fundamental}).
		\par\indent
		The following is the main theorem of this paper, which is a dual version of Theorem \ref{calculational.base}. Note that $j_{\sharp} \rmph(X, L)$ and $j_{\sharp} \rmsph(X, L)$ are subgroups of $\rmph(X, Y)$ and $\rmsph(X, L)$ respectively (Theorem \ref{fundamental}(2)). Let $\hat{\zbb}$ denote the product ${\prod}_p\, \hat{\zbb}_{p}$ of the $p$-completions of $\zbb$, in which $\zbb$ is diagonally contained. Similarly, let $\check{\zbb}$ denote the product ${\prod}_{p}\, \zbb_{(p)}$ of the $p$-localizations of $\zbb$, in which $\zbb$ is diagonally contained.
		
		\begin{thm}\label{dual.base}
			Let $(X, Y)$ be in $\mathcal Q$. Suppose that there exists a fibration sequence ${L \overset{j}{\longrightarrow} Y \overset{q}{\longrightarrow} Y'}$ with
			$Y'$ nilpotent of finite type and ${[X, \Omega \hat{Y'}] = 0.}$ Then there exist natural split exact sequences of abelian groups given by
			\begin{eqnarray*}
				0 \longrightarrow j_{\scriptscriptstyle\sharp}\rmph(X, L) & \longrightarrow  & \rmph(X, Y)\\ & \longrightarrow & \underset{i>0}{\prod}\ H^i (X ; \pi_{i+1}(Y)/ j_{\scriptscriptstyle\sharp}\pi_{i+1}(L) \otimes \hat{\mathbb Z} / \mathbb Z)  \longrightarrow 0, \\
				0 \longrightarrow j_{\scriptscriptstyle\sharp}\rmsph(X, L) & \longrightarrow & \rmsph(X, Y)\\ & \longrightarrow & \underset{i>0}{\prod}\ H^i (X ; \pi_{i+1}(Y) / j_{\scriptscriptstyle\sharp} \pi_{i+1}(L) \otimes \check{\mathbb Z} / \mathbb Z) \longrightarrow 0.
			\end{eqnarray*}
		\end{thm}
		\par\indent
		Let us recall the generalizations of Miller's theorem \cite{Miller84} and Anderson-Hodgkin's theorem \cite{Anderson68} to obtain many pairs $(X, Y')$ with ${[X, \Omega \hat{Y'}] = 0.}$
		A space whose $i^{\rm th}$ homotopy group is zero for $i \leq n$ and locally finite for $i = n + 1$ is said to be {\it $n\frac{1}{2}$-connected}.
		Define the classes $\mathcal A$, $\mathcal B$, $\mathcal A'$ and
		$\mathcal B'$ by
		\begin{enumerate}
			\item[$\mathcal A$ =] the class of $\frac{1}{2}$-connected Postnikov spaces, the classifying spaces of compact Lie groups, 
			$\frac{1}{2}$-connected infinite loop spaces and their iterated suspensions.
			
			\item[$\mathcal B$ =] the class of nilpotent finite complexes, the classifying spaces of compact Lie groups and their iterated loop spaces.
			
			\item[$\mathcal A'$ =] the class of ${1}\frac{1}{2}$-connected
			Postnikov spaces of finite type and their iterated suspensions.
			
			\item[$\mathcal B'$ =] the class of $BU$, $BO$, $BSp$, $BSO$, $U/Sp$,
			$Sp/U$, $SO/U$, $U/SO$, and their iterated loop spaces.
		\end{enumerate}
		If $(X, Y')$ is in $\acal \times \bcal$ or $\acal^{'} \times \bcal^{'}$, then ${[X, \Omega \hat{Y'}] = 0}$ (\cite[Corollary 6.4]{phantom}). 
		
		We have the following corollaries to Theorem \ref{dual.base}. Let $K\langle n \rangle$ denote the $n$-connected cover of $K$.
		\begin{cor}\label{positiveintegerm}
			Let $(X,Y')$ be in $\mathcal{A}\times\mathcal{B}$ or $\mathcal{A}'\times\mathcal{B}'$, and let $m$ be a positive integer. Suppose that $X$ is a $CW$-complex of finite type and that $(X,Y'\langle m\rangle)$ is in $\mathcal{Q}$. Then there exist natural isomorphisms of groups
			\begin{eqnarray*}
				\rmph(X, Y'\langle m\rangle) &\cong& \underset{i>0}{\prod}\ H^i (X; \pi_{i+1}(Y'\langle m\rangle) \otimes \hat{\mathbb Z} / \mathbb Z), \\
				\rmsph(X, Y'\langle m\rangle ) &\cong& \underset{i>0}{\prod}\ H^i (X; \pi_{i+1}(Y'\langle m\rangle) \otimes \check{\mathbb Z}/{\mathbb Z}).
			\end{eqnarray*}
		\end{cor}
		\def\para{%
			\setlength{\unitlength}{1pt}%
			\thinlines %
			\begin{picture}(8,0 )%
			\put(1,0){/}
			\put(3,0){/}
			\end{picture}%
		}%
		Let $L$ be a pointed $CW$-complex endowed with an action of a compact Lie group $H$. Defining the {\em homotopy quotient} $L \para H$ by $L \para H = EH \underset{H}{\times} L$, we have the fiber bundle
		\[
		L \xrightarrow{ j } L \para H \rightarrow BH,
		\]
		where $EH \rightarrow BH$ is the universal principal $H$-bundle. If an injective homomorphism $H \longhookrightarrow \gbb$ of topological groups is given, then $\gbb \para H$ is usually denoted by $\gbb / H$.
		\par\indent
		The following corollary is derived using a result of Atiyah-Segal \cite{AS}.
		\begin{cor}\label{cor1.7}
			Let $X$ be the classifying space $BG$ of a compact Lie group $G$ or its iterated suspension. Let $\gbb$ denote the infinite unitary group $U$ or the infinite orthogonal group O, and let $H$ be a compact Lie group which is a topological subgroup of $\gbb$. Then $(X, \gbb / H)$ is in $\qcal$ and there exist natural isomorphisms of groups
			\begin{eqnarray*}
				\rmph(X, \mathbb{G}/H) \cong \underset{i>0}{\prod}\ H^i (X ; \pi_{i+1}(\mathbb{G}/H) / j{\scriptscriptstyle\sharp}\pi_{i+1}(\mathbb{G}) \otimes \hat{\mathbb Z} / \mathbb Z), \\
				\rmsph(X, \mathbb{G}/H) \cong \underset{i>0}{\prod}\ H^i (X ; \pi_{i+1}(\mathbb{G}/H) / j{\scriptscriptstyle\sharp}\pi_{i+1}(\mathbb{G}) \otimes \check{\mathbb Z} / \mathbb Z).
			\end{eqnarray*}
		\end{cor} 
		Further applications of Theorem \ref{dual.base} are given in Section 2 (see Examples \ref{exa 2.3} - \ref{cor1.5} and Remark \ref{nontrivial2}).
		
		\begin{rem}\label{introend}  
			Most calculations of $\rmph(X, Y)$ have assumed that $Y$ is a nilpotent finite complex or its iterated loop space (\cite[Theorem A and Remark 2.6]{phantom}). Corollaries \ref{positiveintegerm}-\ref{cor1.7} and Examples \ref{exa 2.3}-\ref{cor1.5}, and Remark \ref{contractible} give calculational results for $(X, Y)$ such that $Y$ is not in $\bcal$ or $\bcal'$.\\
		\end{rem}
		\section{Proofs of main results}\label{section2}
		In this section, we prove Theorem \ref{dual.base} and Corollaries \ref{positiveintegerm}-\ref{cor1.7}, and then give further applications of Theorem\ref{dual.base}.
		\par\indent
		We begin by recalling the basic results on $\rmph(X, Y)$ and $\rmsph(X, Y)$.
		\subsection{Groups of homotopy classes of phantom maps}
		In this subsection, we make a review of the main results of \cite{phantom}.
		\par\indent
		Recall the definition of the full subcategory $\qcal$ of $\ccal\wcal^{\mathrm{op}} \times \ncal$ from Section 1. A pair $(X, Y) \in C\wcal^{\rm op} \times \ncal$ is in $\qcal$ if $X$ is a co-$H_{0}$-space or $Y$ is an $H_{0}$-space. $\qcal$ contains many other pairs (see \cite[Section 4.2]{phantom}). 
		\par\indent
		The following theorem, which is Theorem 2.3 in \cite{phantom}, is a fundamental result on group structures on $\rmph(X, Y)$ and $\rmsph(X, Y)$.
		\begin{thm} \label{fundamental}
			Let $(X,Y)$ be an object of $\mathcal Q$.
			\begin{itemize}
				
				\rmitem{1} $\rmph(X,Y)$ and $\rmsph(X,Y)$ have natural divisible
				abelian group structures, for which $\rmsph(X, Y)$ is a subgroup of $\rmph(X, Y)$.
				\rmitem{2} Let $(f^{\rm op},g):(K,L)\longrightarrow (X,Y)$ be a morphism of $\mathcal{CW}^{\rm op}\times\ncal$. Then, the images $\im\  \rmph(f,g)$ and $\im\  \rmsph(f,g)$ are divisible abelian subgroups of $\rmph(X,Y)$ and $\rmsph(X,Y)$  respectively.
				\rmitem{3} If $X$ is a co-$H$-space or $Y$ is an $H$-space, the group structures on $\rmph(X,Y)$ and $\rmsph(X,Y)$ are compatible with the multiplicative structure on $[X,Y]$.
			\end{itemize}
		\end{thm}
		The following theorem, which is Theorem 2.7 in \cite{phantom}, presents a powerful method for calculating the groups $\rmph(X, Y)$ and $\rmsph(X, Y)$ for $(X, Y) \in \qcal$ with $[X, \Omega\hat{Y}] \neq 0$.
		\par\indent
		Note that in the theorem, $p^{\scriptscriptstyle\sharp}\rmph(K, Y)$ and $p^{\scriptscriptstyle\sharp}\rmsph(K, Y)$ are the subgroups of $\rmph(X,Y)$ and $\rmsph(X,Y)$ (see Theorem \ref{fundamental}(2)).
		\begin{thm}\label{calculational.base}
			Let $(X, Y)$ be in $\mathcal Q$. Let ${X' \overset{i}{\longrightarrow} X \overset{p}{\longrightarrow} K}$ be a cofibration sequence with ${[X', \Omega \hat{Y}] = 0}$, or a fibration sequence with weakly contractible ${{\rm map}_\ast(X', \Omega \hat{Y})}$. Then there exist natural split exact sequences of abelian groups given by
			\begin{eqnarray*}
				0 \longrightarrow p^{\scriptscriptstyle\sharp}\rmph(K, Y) & \longrightarrow  & \rmph(X, Y)\\ & \longrightarrow & \underset{i>0}{\prod}\ H^i (X ; \pi_{i+1}(Y) \otimes \hat{\mathbb Z} / \mathbb Z) / p^\ast H^i(K ; \pi_{i+1}(Y) \otimes \hat{\mathbb Z} / \mathbb Z) \longrightarrow 0, \\
				0 \longrightarrow p^{\scriptscriptstyle\sharp}\rmsph(K, Y) & \longrightarrow & \rmsph(X, Y)\\ & \longrightarrow & \underset{i>0}{\prod}\ H^i (X ; \pi_{i+1}(Y) \otimes \check{\mathbb Z} / \mathbb Z) / p^\ast H^i(K ; \pi_{i+1}(Y) \otimes \check{\mathbb Z} / \mathbb Z) \longrightarrow 0.
			\end{eqnarray*}
		\end{thm}
		
		See \cite[Corollaries 2.8-2.10 and Example 6.6]{phantom} for the applications.
		\subsection{Proofs of Theorem \ref{dual.base} and Corollary \ref{positiveintegerm}}
		For a nilpotent space $Y$ of finite type, the profinite completion $\hat{Y}$ and the local expansion $\check{Y}$ are defined by 
		$\hat{Y} = {\prod}_{p}\, \hat{Y}_p$ and $\check{Y} = {\prod}_{p}\, Y_{(p)}$, respectively, where $\hat{Y}_p$ and $Y_{(p)}$ are the $p$-profinite completion and the $p$-localization of $Y$ respectively (\cite{Sullivan74}). Thus, we can establish a commutative diagram of natural transformations
		\begin{center}
			\begin{tikzcd}
			&  Y  \ar[swap]{ld}{e_{Y}} \ar{rd}{c_{Y}} &  \\
			\check{Y} \ar[swap]{rr}{d_{Y}}  & & \hat{Y}.
			\end{tikzcd}
		\end{center}
		Let $F_{Y}$ denote the homotopy fiber of $c_{Y}$.
		\begin{proof}[{\sl Proof of Theorem \ref{dual.base}}]
			Since the proof of Theorem \ref{dual.base} is similar to that of Theorem \ref{calculational.base}, our proof is sketchy; the details are found in the proof of \cite[Theorem 2.7]{phantom}.
			\par\indent
			By replacing $Y$ and $Y'$ with their universal covers, we may assume that $L$ is in $\mathcal N$ (see \cite[Remark 5.6]{phantom} and \cite[Proposition 4.4.1]{no10}). 
			\par\indent
			{\sl The case of $\rmph(X,Y)$.} By \cite[Corollary 5.3]{phantom} and the comment before \cite[Lemma 3.5]{phantom}, we have the morphism of exact sequences of pointed sets
			\begin{center}	
				\begin{tikzcd}
				{[X, \Omega \hat{L}]}			 \arrow[swap]{d}{{\Omega\hat{j}_{\sharp}}}   \arrow{r} &   {[X, F_L]} \arrow{r} \arrow[swap]{d}{F_{j\sharp}} & \rmph(X, L) \arrow{r} \arrow[swap]{d}{j_{\sharp}} & 0\\
				{[X, \Omega \hat{Y}]} \arrow{r} & {[X, F_{Y}]} \arrow{r} & \rmph(X, Y) \arrow{r} & 0.
				\end{tikzcd}
			\end{center}
			Note that since $[X,\Omega\hat{Y'}]=0$, the left vertical arrow is surjective. Next, consider the induced morphism of exact sequences of pointed sets
			\begin{center}	
				\begin{tikzcd}
				0 \arrow{r} &	\ker\ \alpha			 \arrow[swap]{d}{\beta}   \arrow{r} &   F_{j\sharp}{[X, F_L]} \arrow{r}{\alpha} \arrow[swap]{d} & j_{\sharp}\rmph(X, L) \arrow{r} \arrow[swap]{d} & 0\\
				0 \arrow{r} &	\overline{[X, \Omega \hat{Y}]} \arrow{r} & {[X, F_{Y}]} \arrow{r} & \rmph(X, Y) \arrow{r} & 0,
				\end{tikzcd}
			\end{center}
			where $\overline{[X,\Omega\hat{Y}]}$ denotes the image of $[X,\Omega\hat{Y}]$ and $\alpha$ denotes the map induced by the natural quotient map $[X,F_L]\rightarrow \rmph(X,L)$. Note that this is a morphism of exact sequences of abelian groups (see \cite[Theorem 2.3 and its proof]{phantom}) and that since $\Omega\hat{j}_\sharp:[X,\Omega\hat{L}]\rightarrow[X,\Omega\hat{Y}]$ is surjective, $\beta$ is also surjective. Then, we have
			\[
			\rmph(X, Y) / j_{\sharp} \rmph(X, L) \cong  [X, F_{Y}]/ F_{j\sharp} [X, F_{L}],
			\]
			from which we obtain the desired sequence (see \cite[Proposition 5.10 and the proof of Theorem 2.3(2)]{phantom}).
			\par\indent
			{\sl The case of $\rmsph(X,Y)$.} As mentioned in the introduction of Section 6.1 of \cite{phantom}, $\rmph(X,Y)$ and $\rmph(X,\check{Y})$ generate analogous results. Thus, there exists a morphism of exact sequences of abelian groups
			\begin{center}	
				\begin{tikzcd}
				0 \arrow{r} &	j_{\sharp}\rmph(X, L)\arrow[swap]{d}{\epsilon}   \arrow{r} &  \rmph(X, L)  \arrow{r} \arrow[swap]{d}{e_{Y\sharp}} & \underset{i>0}{\displaystyle\prod }\ H^i (X;\pi_{i+1}(Y)/j_\sharp\pi_{i+1}(L)\otimes\hat{\mathbb{Z}}/\mathbb{Z}) \arrow{r} \arrow[swap]{d} & 0\\
				0 \arrow{r} &	j_{\sharp}\rmph(X, \check{L}) \arrow{r} & \rmph(X, \check{Y}) \arrow{r} & \underset{i>0}{\displaystyle\prod}\ H^i (X;\pi_{i+1}(Y)/j_\sharp\pi_{i+1}(L)\otimes\hat{\mathbb{Z}}/\check{\mathbb{Z}}) \arrow{r} & 0,
				\end{tikzcd}
			\end{center}
			which gives rise to the exact sequence
			\[
			0 \longrightarrow \ker\ \epsilon  \longrightarrow \mathrm{SPh}(X, Y) \longrightarrow \underset{i > 0}{\prod}H^{i}(X; \pi_{i+1}(Y)/j_{\sharp}\pi_{i+1}(L) \otimes \check{\zbb}/\zbb ) \longrightarrow 0.
			\]
			An argument similar to that in the proof of \cite[Theorem 2.7]{phantom} shows that $\ker\, \epsilon \cong j_{\sharp} \mathrm{SPh}(X, L)$.
		\end{proof}
		\begin{proof}[{\sl Proof of Collorary \ref{positiveintegerm}}]
			We have the fibration sequence
			\[
			\Omega Y^{'(m)} \to Y'\langle m\rangle \to Y', 
			\]
			where $Y^{'(m)}$ is the Postnikov $m$-stage of $Y'$. Since $X$ is a $CW$-complex with finite skeleta, $\rmph(X,\Omega Y^{'(m)})$, and hence $\rmsph(X,\Omega Y^{'(m)})$ vanishes. Thus, the result follows from Theorem \ref{dual.base}.
		\end{proof}
		\subsection{Proof of Corollary \ref{cor1.7} and further applications}
		For the proof of Corollary \ref{cor1.7}, we prove the following two lemmas, which are interesting in their own right.
		\begin{lem}\label{adjoint}
			Let $X$ and $Y$ be connected $CW$-complexes and $K$ be a finite complex. Then there exists a natural isomorphism
			\begin{eqnarray*}
				\rmph(K {\scriptstyle \wedge} \hspace{0.4mm}X, Y) & \cong & \rmph(X, {\rm map}_*(K, Y))
			\end{eqnarray*}
			and if $Y$ is nilpotent of finite type, then there also exists a natural isomorphism
			\begin{eqnarray*}
				\rmsph(K {\scriptstyle \wedge} \hspace{0.4mm}X, Y) & \cong & \rmsph(X, {\rm map}_*(K, Y)). \\
			\end{eqnarray*}
		\end{lem}
		
		\begin{proof}
			Note that $f : K {\scriptstyle \wedge} \hspace{0.4mm}X \longrightarrow Y$ is phantom if and only if $f \mid_{K {\scriptstyle \wedge} \hspace{0.3mm}X_\alpha}$ is null homotopic for any finite subcomplex $X_\alpha$ of $X$.
			Then the natural isomorphism $[K {\scriptstyle \wedge} \hspace{0.4mm}X, Y] \cong [X, {\rm map}_*(K, Y)]$ is clearly restricted to the first natural isomorphism. If $Y$ is nilpotent of finite type, we obtain the isomorphism
			\[
			\rmph(K {\scriptstyle \wedge} \hspace{0.4mm}X, \check{Y})  \cong  \rmph(X, {\rm map}_*(K, \check{Y}))
			\]
			(see the introduction of Section 6.1 of \cite{phantom}), which also implies the second isomorphism by \cite[Theorem 6.3.2]{no10} and the definition of $\rmsph(X, Y)$.  
		\end{proof}
		\begin{lem}\label{AScompletion}
			Let $G$ be a compact Lie group and $X$ a finite $G$-$CW$-complex. Then,
			\[
			\rmph(X \para\, G, \Omega^{l}U) = 0 \ \text{and} \ \rmph(X \para\, G, \Omega^{l}O) = 0 
			\]
			hold for $l \geq 0$.
		\end{lem}
		\begin{proof}
			First, we show the first vanishing result. Since $\Omega^{l} U$ is an $H$-space, we have only to consider the homotopy classes of maps from $X \para G$ to the identity component of $\Omega^{l}U$ in unbased context (\cite[Proposition 1.4.3]{no10}). Note that $K^{\ast}_{G}(X)$ is finite over $R(G)$. Then, the result follows from the proof of \cite[Proposition 4.2]{AS}.
			\par\indent
			The second vanishing result can be similarly proved using the $KO$-versions of the results of \cite{AS}, which are established in \cite{AHJM} (see the comment after Theorem 1.1 in \cite{AHJM} and Remark \ref{completion}(2)).
		\end{proof}
		
		\begin{rem}\label{completion} (1) See \cite[Remark 5.1]{AHJM} for other $G$-spaces for which the vanishing results in Lemma \ref{AScompletion} hold.\\
			(2) The article \cite{AS} deals with not $KO$ but $KR$ as the real case.
			
		\end{rem}

		\begin{proof}[{\sl Proof of Collorary \ref{cor1.7}}]
			By the assumption, $X = \Sigma^{l} BG$ for $l \geq 0$. To prove that $(X, \gbb/ H)$ is in $\qcal$, we have only to show that $(BG, \gbb/H)$ is in $\qcal$, which is easily seen from \cite[Example 4.6(2)]{phantom}. Consider the fibration sequence
			\[
			\gbb \xrightarrow{j} \gbb/H \rightarrow BH
			\]
			and note that $(X, BH)$ is in $\acal \times \bcal$. Since the identities
			\[
			\rmph(X, \gbb) = \rmph(BG, \Omega^{l} \gbb) = 0
			\]
			hold (Lemmas \ref{adjoint} and \ref{AScompletion}), we obtain the result by Theorem \ref{dual.base}.
		\end{proof}
		We give further applications of Theorem \ref{dual.base}. For this, we show the following lemma, which is useful to find many pairs $(X, L)$ with $\rmph(X, L) = 0$. 
		\par\indent
		A space whose $i^{\rm th}$ $k$-invariant vanishes for all but finite $i$ is called a {\it generalized Postnikov space.}
		\begin{lem}\label{lem2.2}
			$\mathrm{(1)}$ If $\{i \ > \ 0 \mid H^i(A; \mathbb Q) \neq 0 \} \ {\textstyle\bigcap} \ \{j \ > \ 0 \mid \pi_{j+1}(B) \otimes \mathbb Q \neq 0 \} = \emptyset$, then $\rmph(A, B) = 0$.
			\rmitem{2}				
			If $A$ is a $CW$-complex of finite type and $B$ is a generalized Postnikov space, then $\rmph(A, B) = 0$.
		\end{lem}
		\begin{proof}
			(1) The result follows from \cite[Propositions 5.7 and 4.1]{phantom}.\\
			(2) By the finite type assumption on $A$, all elements of $\rmph(A, B)$ are skeletally phantom and $B$ is homotopy equivalent to the product of the Postnikov $n$-stage $B^{(n)}$ and ${\prod}_{i>n}\, K(\pi_i(B), i)$ for sufficiently large $n$ (see \cite[Remark 3.3]{phantom}). Therefore, $\rmph(A, B) = \rmph(A, B^{(n)}) \times {\prod}_{i>n}\, \rmph(A, K(\pi_i(B), i))$ vanishes.
		\end{proof}
		For a connected $CW$-complex $K$, $Q(K)$ denotes the infinite loop space defined by $Q(K) = \underset{n}{\mathrm{colim}}\, \Omega^{n} \Sigma^{n}K.$ 
		\begin{exa}\label{exa 2.3} 
			Let $\gbb$ and $H$ be as in Corollary \ref{cor1.7}. Noticing that $Q(S^{2n+1})$ is rationally equivalent to $S^{2n+1}$, we calculate the groups $\rmph(Q(S^{2n+1}), \gbb/ H)$ and $\rmsph(Q(S^{2n+1}), \gbb / H)$ for $n \geq 1$. Consider the fibration sequence
			\[
			\gbb \xrightarrow{j} \gbb/H \rightarrow BH.
			\]
			Since $(Q(S^{2n+1}), BH )$ is in $\acal \times \bcal$ and $\rmph(Q(S^{2n+1}), \gbb) = 0$ (Lemma \ref{lem2.2}(1)), we have the isomorphisms of abelian groups
			\begin{eqnarray*}
				\rmph(Q(S^{2n+1}), \gbb/H) & \cong & \pi_{2n+2} (\gbb/H) /j_{\sharp} \pi_{2n+2} (\gbb) \otimes \hat{\zbb}/ \zbb, \\
				\rmsph(Q(S^{2n+1}), \gbb/H) & \cong & \pi_{2n+2} (\gbb/H) /j_{\sharp} \pi_{2n+2} (\gbb) \otimes \check{\zbb}/ \zbb
			\end{eqnarray*}
			by Theorem \ref{dual.base}.
		\end{exa}
		Recall that for a $CW$-complex $L$ with an $H$-action, we have the fiber bundle
		\[
		L \xrightarrow{\ \ \ j\ \ \ } L \para H \xrightarrow{\ \ \ \ \ \ } BH
		\]
		(see Section 1).
		\begin{exa}\label{cor1.5}
			
			Let $X$ be a $CW$-complex of finite type which is in $\acal$ and let $L$ be a generalized Postnikov space with an action of a compact Lie group $H$. Suppose that $(X, L \para H)$ is in $\qcal$. Then there exist natural isomorhpisms of groups
			\begin{eqnarray*}
				\rmph(X, L\para H) \cong \underset{i>0}{\prod}\ H^i (X ; \pi_{i+1}(L\para H) / j{\scriptscriptstyle\sharp}\pi_{i+1}(L) \otimes \hat{\mathbb Z} / \mathbb Z), \\
				\rmsph(X, L\para H) \cong \underset{i>0}{\prod}\ H^i (X ; \pi_{i+1}(L\para H) / j{\scriptscriptstyle\sharp}\pi_{i+1}(L) \otimes \check{\mathbb Z} / \mathbb Z).
			\end{eqnarray*}
			(Theorem \ref{dual.base} and Lemma \ref{lem2.2}(2).)
			\par\indent
			Let $L$ be the infinite symmetric product $SP(M)$ of a $CW$-complex $M$ endowed with an action of a compact Lie group $H$. Since $L$ is weak equivalent to the product of Eilenberg-MacLane complexes, this result is applicable to $L = SP(M)$. This result is also applicable to the case where $L$ is an Eilenberg-MacLane $H$-space (\cite[p. 21]{MPC}).
		\end{exa}

		\begin{rem}\label{nontrivial2}
			Under the assumptions of Theorem \ref{dual.base}, there exist (noncanonical) isomorphisms of abelian groups
			\begin{eqnarray*}
				\rmph(X, Y) \cong j_{\sharp}\rmph(X, L) \oplus \underset{i>0}{\prod}\ H^i(X; \pi_{i+1}(Y)/j_{\sharp} \pi_{i+1} (L) \otimes \hat{\zbb}/\zbb), \\
				\rmsph(X, Y) \cong j_{\sharp}\rmsph(X, L) \oplus \underset{i>0}{\prod}\ H^i (X; \pi_{i+1}(Y)/j_{\sharp} \pi_{i+1} (L) \otimes \check{\zbb}/\zbb).
			\end{eqnarray*}
			Thus, we can obtain nontriviality results of $\rmph(X, Y)$ and $\rmsph(X, Y)$ via rational homotopical computations, even if we do not know whether $j_{\sharp} \rmph (X, L)$ is nontrivial.
		\end{rem}
		\begin{rem}\label{contractible} In this remark, we consider situations similar to those described in Corollary \ref{cor1.7} and Example \ref{cor1.5}, in which the groups $\rmph(X, Y)$ and $\rmsph(X, Y)$ are calculated using not Theorem \ref{dual.base} but \cite[Proposition 6.1]{phantom}.\\
			(1) Let $X$ be a space in $\acal'$ and let $\gbb$ and $H$ be as in Corollary \ref{cor1.7}. Suppose that $(X, \gbb/ H)$ is in $\qcal$. Then, we can calculate the groups $\rmph(X, \gbb/H)$ and $\rmsph(X, \gbb/H)$ by \cite[Proposition 6.1]{phantom}.\\
			(2) Let $X$ be a space in $\acal$ and $L$ be a nilpotent finite complex endowed with an action of a compact Lie group $H$. Suppose that $(X, L \para H )$ is in $\qcal$. Then, we can calculate the groups $\rmph(X, L \para H)$ and $\rmsph(X, L \para H)$ by \cite[Proposition 6.1]{phantom}.

		\end{rem}

\end{document}